\documentclass[a4paper,12pt]{amsart}
\usepackage[T1]{fontenc}
 \usepackage{graphicx}
\usepackage{amsmath,amssymb,mathrsfs}
\usepackage{latexsym}
\usepackage{amsfonts}
\usepackage[latin1]{inputenc}
\usepackage{amssymb,amsbsy,amsmath,amsfonts,amscd}

\allowdisplaybreaks[1]

\newtheorem{thm}{Theorem}[section]
\newtheorem{rema}{Remark}[section]

\newtheorem{cor}[thm]{Corollary}
\newtheorem{lem}[thm]{Lemma}
\newtheorem{prop}[thm]{Proposition}
\newtheorem{conj}[thm]{Conjecture}
\newtheorem{defn}[thm]{Definition}
\theoremstyle{definition}

\numberwithin{equation}{section}

\newcommand{\C}{\mathbb C}

\newcommand{\R}{\mathbb R}

\renewcommand{\b}{\beta}

\newcommand{\frenchresumename}{R\'esum\'e}

\newcommand{\resumename}{Abstract}

\everymath{\displaystyle}



\thanks{This work was supported by The Hubert Curien-Utique contracts 06/S1502 and 09/G 1502.\\
I thank my professor Didier Arnal for his aid an i thank the University of Bourgogne for its hospitality during my stay in France.}

\title[Overalgebras and separation of generic coadjoint orbits of $SL(n, \R)$]{Overalgebras and separation of generic coadjoint orbits of $SL(n, \R)$}

\author[ A. Zergane]{ Amel Zergane $^{\star}$ }

\address{$^{\star}$
Institut de Math\'ematiques de Bourgogne\\
UMR CNRS 5584\\
Universit\'e de Bourgogne\\
U.F.R. Sciences et Techniques
B.P. 47870\\
F-21078 Dijon Cedex \\France}
\address{$^{\star}$ Laboratoire de Math\'ematique Physique, Fonctions Sp\'eciales et Applications, Universit\'e de Sousse\\
Ecole Sup\'erieure des Sciences et de Technologie de Hammam-Sousse\\
Rue Lamine Abassi 4011 H. Sousse\\
Tunisie} \email{amel.zergane@u-bourgogne.fr}
\email{amel.zergan@yahoo.fr}

\begin{document}


\begin{abstract}
For the semi simple and deployed Lie algebra $\mathfrak g=\mathfrak{sl}(n, \R)$, we give an explicit construction of an overalgebra $\mathfrak g^+=\mathfrak g\rtimes V$ of $\mathfrak g$, where $V$ is a finite dimensional vector space. In such a setup, we prove the existence of a map $\Phi$ from the dual $\mathfrak g^\star$ of $\mathfrak g$ into the dual $(\mathfrak g^+)^\star$ of $\mathfrak g^+$ such that the coadjoint orbits of $\Phi(\xi)$, for generic $\xi$ in $\mathfrak g^\star$, have a distinct closed convex hulls. Therefore, these closed convex hulls separate 'almost' the generic coadjoint orbits of $G$.

\end{abstract}
           

\maketitle

\section{Introduction}


In this paper, we prove that, for $n>2$, the Lie algebra $\mathfrak g=\mathfrak{sl}(n,\R)$ admits an overalgebra almost separating of dgree $n$, but $\mathfrak g$ does not admit an overalgebra of degree $2$. More precisely:

There exist a Lie overalgebra $\mathfrak g^+=\mathfrak{sl}(n,\R)\rtimes V$ and an application $\Phi$ of degree $n$, $\Phi: \mathfrak g^\star\longrightarrow\mathfrak g^{+\star}$ such that:

\begin{itemize}
\item[1.] $p\circ\Phi=id_{\mathfrak g^\star}$, where $p$ is the canonical projection $p:\mathfrak g^{+\star}\longrightarrow\mathfrak g^\star$,

\item[2.] $\Phi(Coad(SL(n,\R))\xi)=Coad(G^+)\Phi(\xi)$,
\item[3.] if $\xi$ is generic, then $\overline{Conv}\left(\Phi(Coad(SL(n,\R)\xi))\right)=\overline{Conv}\left(\Phi(Coad(SL(n,\R))\xi')\right)$ if and only if $Coad(SL(n,\R))\xi'$ belongs to a finite set of coadjoint orbits of $\mathfrak{sl}(n,\R)$ (here: a singleton if $n$ is odd, a singleton or a set of two elements if $n$ is even).
\end{itemize}

We identify $(\mathfrak g^+)^\star$ the dual of $\mathfrak g^+$ with the space $\mathfrak g^\star\oplus V^\star$. The condition 1. means $\Phi(\xi)=\xi+\phi(\xi)$, where $\phi$ is a polynomial of degree $n$ from $\mathfrak g^\star$ to $V^\star$. We say that $(\mathfrak g^+,\phi)$ is an overalgebra almost separating of $\mathfrak g$ (of degree $n$).

But there is no separating overalgebra of degree 2, $(\mathfrak g^+_2,\phi)$, i.e there is neither a Lie overalgebra $\mathfrak g_2^+=\mathfrak{sl}(n,\R)\rtimes V_2$ nor $\phi:\mathfrak g^\star\longrightarrow V^{\star}_2$ of degree 2 such that :
\begin{itemize}
\item[1.] $p\circ\Phi=id_{\mathfrak g^\star}$, if $p$ is the canonical projection $p:\mathfrak g^{+\star}_2\longrightarrow\mathfrak g^\star$,
\item[2.] $\Phi(Coad(SL(n,\R))\xi)=Coad(G^+)\Phi(\xi)$,
\item[3.] if $\xi$ is generic then, $\overline{Conv}\left(\Phi(Coad(SL(n,\R)\xi))\right)=\overline{Conv}\left(\Phi(Coad(SL(n,\R))\xi')\right)$ if and only if $Coad(SL(n,\R))\xi'$ belongs to a finite family of coadjoint orbits.
\end{itemize}

Finally, we show that $\mathfrak{sl}(4,\R)$ does not admit an overalgebra almost separating of degree 3.\\


\section{Description of orbits of $\mathfrak{sl}(n,\R)$}


\subsection{Invariant functions of $\mathfrak{sl}(n,\R)$}

 
\

Since the Lie algebra  $\mathfrak{sl}(n,\R)$ is simple, we can identify the adjoint action and coadjoint action of the Lie group $G=SL(n,\R)$. More precisely, we consider the non degenerate  bilinear invariant form on $\mathfrak{sl}(n,\R)$ defined by:
$$
\langle X,Y\rangle=Tr(XY).
$$
Denote $\xi$ an element of the dual of $\mathfrak{sl}(n,\R)$ and $X$ an element of $\mathfrak{sl}(n,\R)$.\\
The functions defined on the dual $\mathfrak g^*=\mathfrak{sl}(n,\R)$ of the Lie algebra $\mathfrak g $ by:
$$
T_k(\xi)=Tr(\xi^k),\quad 2\leq k \leq n
$$
are invariant: $T_k(g\xi g^{-1})=T_k(\xi)$ for all $\xi$ and all $g$.\\ 

The ring of the polynomial invariant functions on $\C$ is $\C[T_2,\dots, T_n]$ ({\sl{cf.}} \cite{W}).\\


\subsection{Description of $\Omega=\{\xi\in\mathfrak{sl}(n,\R),\#Sp(\xi)=n\}$}

\


\

This is classic. We recall this only for completeness.

Denote $\Omega$ the set of matrices $\xi$ of $\mathfrak{sl}(n,\R)$ which are diagonalizable on $\C$ and have $n$ distinct eigenvalues. The spectrum of this matrix $\xi$ is 
$$
Sp(\xi)=\{c_1,\dots,c_r,a_1\pm ib_1,\dots,a_s\pm ib_s,\quad\text{with}~~c_i,~a_j,~b_j\in\R,~~b_j>0,~~r+2s=n\}.
$$

Denote by $(a,b)<(a',b')$ the lexicographic order :
$$
(a,b)<(a',b')~\Longleftrightarrow~\left\{\aligned &a<a'\\ \text{or}&\\ &a=a'~\text{ and }~b<b'.\endaligned\right.
$$
We note also (the same if the eigenvalues are not stored in lexicographic order ) :
$$
D(c_1,\dots,c_r,a_1+ib_1,\dots,a_s+ib_s)=\left(\begin{matrix}
c_1&&&&&&&\\
&\ddots&&&&&&\\
&&c_r&&&&&&\\
&&&a_1&b_1&&&\\
&&&-b_1&a_1&&&\\
&&&&&\ddots&&\\
&&&&&&a_s&b_s\\
&&&&&&-b_s&a_s\\
\end{matrix}\right).
$$

We fix $r$ and $s$ such that $r+2s=n$. If $r>0$, we put :
$$
\Sigma_{r,s}=\{D(c_j,a_k+ib_k),~~c_1<c_2<\dots<c_r,~~b_k>0,~~(a_1,b_1)<(a_2,b_2)<\dots<(a_s,b_s)\} 
$$
If $r=0$, we note :
$$
\Sigma^+_{0,s}=\{D(a_k+ib_k),~~b_k>0,~~(a_1,b_1)<(a_2,b_2)<\dots<(a_s,b_s)\} 
$$
and
$$
\Sigma^-_{0,s}=\{D^-(a_k+ib_k),~~b_k>0,~~(a_1,b_1)<(a_2,b_2)<\dots<(a_s,b_s)\},
$$
where
$$
D^-(a_1+ib_1,\dots,a_s+ib_s)=\left(\begin{matrix}
a_1&-b_1&&&&&\\
b_1&a_1&&&&&\\
&&a_2&b_2&&&\\
&&-b_2&a_2&&\\
&&&&\ddots&&\\
&&&&&a_s&b_s\\
&&&&&-b_s&a_s\\
\end{matrix}\right).
$$
Finally, we put
$$
\Sigma=\left\{\aligned&\bigcup_{r>0,~r+2s=n}\Sigma_{r,s}\bigcup\left(\Sigma^+_{0,\frac{n}{2}}\cup \Sigma^-_{0,\frac{n}{2}}\right)&~~\text{ if }~~n~\text{ is even}\\
&\bigcup_{r>0,~r+2s=n}\Sigma_{r,s}&~~\text{ if }~~n~\text{ is odd}.\endaligned\right.
$$

The set $\Omega$ is invariant since the spectrum $Sp(\xi)$ of $\xi$ coincides with the spectrum of $g\xi g^{-1}=Ad(g)(\xi)$  ($g\in SL(n,\R)$), or if we prefer, if $C_\xi$ is the characteristic polynomial of the matrix $\xi$, then the adjoint orbit $G\cdot\xi_0$ of $\xi_0$ is included in $\{\xi,~~\text{such that}~~C_\xi=C_{\xi_0}\}$.


\subsection{Adjoint orbits in $\Omega$}

\


\

\begin{lem}
\

For all matrix $\xi$ in $\Omega$, the adjoint orbit $G\cdot\xi$ of the matrix $\xi$ contains a point of $\Sigma$.
\end{lem}
\begin{proof}
\

Let $\xi$ be a matrix in $\Omega$,  $\xi$ is diagonalizable on $\C$, with eigenvalues all distinct. If $Sp(\xi)=\{c_j,a_k\pm ib_k\}$, where the eigenvalues are ordered as above, $c_1<\dots<c_r$ and $(a_1,b_1)<\dots<(a_s,b_s)$, $b_1>0,\dots,b_s>0$, then there exist vectors $E_j\in\C^n$ and $F_k\in\C^n$ such that $\xi E_j=c_j E_j$ and $\xi F_k=(a_k+ib_k)F_k$.

Since $\xi$ is real, we can choose $E_j$ real (in $\R^n$) and if we put $F_k=E_{r+2k-1}+iE_{r+2k}$ ($E_{r+t}$ are real), then we obtain a basis of $\R^n$. If $P$ is the basis change matrix , then the matrix of $\xi$ is written in the new basis as follows : 
$$
\xi'=P\xi P^{-1}=D(c_1,\dots,c_r,a_1+ib_1,\dots,a_s+ib_s).
$$
\begin{itemize} 
\item[a.] If $\det P>0$, then there exists $g=\frac{1}{\sqrt[n]{\det P}}P$, such that $\xi'=g\xi g^{-1}$ and $g\in SL(n,\R)$. The adjoint orbit $G\cdot\xi$ of the matrix $\xi$ contains a point of $\Sigma$.
\item[b.] If $\det P<0$ and $r>0$, then we replace $E_1$ by $E'_1=-E_1$. The matrix $P$ becomes $P'=D(-1,1,\dots,1)P$, $P\xi P^{-1}=P'\xi P'^{-1}$ and $\det(P')>0$. As above, the adjoint orbit $G\cdot\xi$ of the matrix $\xi$ contains a point of $\Sigma$.
\item[c.] If $\det P<0$ and $r=0$, then we replace $E_1$ by $E'_1=E_2$ and $E_2$ by $E'_2=E_1$. The matrix $P$ becomes $P'=D(-1,1,\dots,1)P$, $P\xi P^{-1}=P'\xi P'^{-1}$ and $\det(P')>0$. The adjoint orbit $G\cdot\xi$ of the matrix $\xi$ contains the point $\xi'=P'\xi P'^{-1}=D^-(a_1+ib_1,\dots,a_s+ib_s)$ of $\Sigma$.
\end{itemize}
\end{proof}

\begin{lem}
\

$\Sigma$ is a section for the action of $SL(n,\R)$ in $\Omega$, i.e each orbit of $\Omega$ contains only a single point of $\Sigma$.
\end{lem}
\begin{proof}

\

Let $\xi_0$ be an element in $\Sigma$ and $G\cdot\xi_0$ its orbit. If $\xi\in G\cdot\xi_0\cap\Sigma$, then,  since the spectrum of $\xi$ is the same as $\xi_0$ and the order of eigenvalues is fixed, we obtain:

If $r>0$ then $\xi=\xi_0$. If $r=0$, we have either $\xi=\xi_0$, or 
$$
G\cdot\xi_0\cap\Sigma=\{\xi,\xi_0\}=\{D^+(a_k+ib_k),D^-(a_k+ib_k)=gD^+(a_k+ib_k)g^{-1}\}.
$$
with $\det(g)=1$. 

In the latter case, the sub eigenspaces $Vec_\C(E_{2k-1}+iE_{2k})$ and $Vec_\C(E_{2k-1}-iE_{2k})$ are one dimensional. Thus, there exist a nonzero complex numbers $z_1,\dots,z_s$ such that : 
$$
g(E_1+iE_2)=z_1(E_1-iE_2),\quad\text{and}\quad g(E_{2k-1}+iE_{2k})=z_k(E_{2k-1}+iE_{2k}),\quad k>1.
$$
The matrix of $g$ is written in the basis of eigenvectors of the first matrix as follows :
$$
QgQ^{-1}=\left(\begin{matrix}0&\overline{z_1}&&&&\\
z_1&0&&&&&\\
&&z_2&&&&\\
&&&\overline{z_2}&&&\\
&&&&\ddots&&\\
&&&&&z_s&\\
&&&&&&\overline{z_s}\\
\end{matrix}\right).
$$
The determinant of the matrix $QgQ^{-1}$ is negative or zero, which is impossible. Therefore, $\xi=\xi_0$ and $\Sigma$ is a section for the action of $SL(n,\R)$ on $\Omega$.
\end{proof}

\begin{lem}

\

Denote $\Omega_{r,s}=G\cdot \Sigma_{r,s}$. Let $\xi_0\in\Omega_{r,s}$.
\begin{itemize}
\item[1.] If $r>0$, $\{\xi,$ such that $C_\xi=C_{\xi_0}\}$ is exactly the adjoint orbit $G\cdot\xi_0$ of $\xi_0$.\\
\item[2.] If $r=0$, $\{\xi,$ such that $C_\xi=C_{\xi_0}\}$ is the union of two adjoint orbits $G\cdot\xi_0\sqcup G\cdot\xi_1$ .\\
\end{itemize}
\end{lem}

Now, show that $\Omega$ is dense. Let $\xi$ be an arbitrary matrix of $\mathfrak{sl}(n,\R)$. On $\C$, we can transform this matrix in Jordan form:  
\begin{itemize}
\item[1.] for each Jordan block $J_j(c)$ associated to the real eigenvalue $c$ of $\xi$, there exist vectors $E_j^1,\dots,E_j^{t}$ in $\C^n$ such that $\xi E^1_j=c E^1_j$ and, if $t>1$, $\xi E^t_j=c E_j^t+ E_j^{t-1}$. We can choose $E_j^t$ real.\\

\item[2.] for each Jordan block $J_k(a+ib)$ associated to the non real eigenvalue $a+ib$, with $b>0$, we can choose vectors $F_k^1,\dots,F_k^{t}$ in $\C^n$ such that $\xi F^1_k=(a+ib) F^1_k$ and, if $t>1$, $\xi F^t_k=(a+ib) F_k^t+ F_k^{t-1}$. The union of these families of vectors, for all Jordan blocks associated to $a+ib$, form a basis of the characteristic subspace $V(a+ib)$ of $\xi$ associated to $a+ib$. The combined vectors $\overline{F_k^t}$ form a basis of the characteristic subspace $V(a-ib)$. In this basis, the matrix of $\xi|_{V(a-ib)}$ is also in Jordan form. As above, put $F_k^t=E_k^t+iE'^t_k$, where $E_k^t$ and $E'^t_k$ are real vectors.\\

\end{itemize}
We arrange the eigenvalues of $\xi$ as above, then we obtain a real basis of $\R^n$. On this basis, $\xi$ is written as follows :
$$\aligned
\xi'&=P\xi P^{-1}=\left(\begin{matrix}c_1&u_1&&&&&&&&\\ &\ddots&\ddots&&&&&&&\\ &&c_{r-1}&u_{r-1}&&&&&&\\ &&&c_r&&&&&&\\ &&&&a_1&b_1&v_1&&&\\ &&&&-b_1&a_1&0&v_1&&\\&&&&&&\ddots&\ddots&&\\ &&&&&&a_{s-1}&b_{s-1}&v_{s-1}&\\ &&&&&&-b_{s-1}&a_{s-1}&0&v_{s-1}\\ &&&&&&&&a_s&b_s\\ &&&&&&&&-b_s&a_s\end{matrix}\right)\\
&=Diag(c_1,\dots,c_r,\left(\begin{matrix}a_1&b_1\\-b_1&a_1\end{matrix}\right),\dots, \left(\begin{matrix}a_s&b_s\\-b_s&a_s\end{matrix}\right))+\\
&\hskip 1cm+ Overdiag(u_1,\dots,u_{r-1},\left(\begin{matrix}v_1&0\\0&v_1\end{matrix}\right),\dots,\left(\begin{matrix}v_{s-1}&0\\0&v_{s-1}\end{matrix}\right))\\
&=D+N.
\endaligned
$$
Where $P$ is a real matrix, $Diag$ means that we place the cited sub matrices on the diagonal, $Overdiag$ means that we place the cited sub matrices on the second diagonal,  $u_j$ and $v_k$ are either 0 or 1.

Let now the real numbers $x_1,\dots,x_r$ and $y_1,\dots,y_s$  all distinct such that $x_1+\dots+x_r+2y_1+\dots+2y_s=0$. Put:
$$
A=Diag(x_1,\dots,x_r,y_1,y_1,\dots,y_s,y_s),
$$
and for all $\varepsilon>0$, $\xi_\varepsilon=\xi+\varepsilon P^{-1}AP$. For almost every $\varepsilon$, the trace of $\xi_\varepsilon$ is zero and $\xi_\varepsilon$ has $n$ distinct eigenvalues. Then, $\xi_\varepsilon\in\Omega$ and for any norm on $\mathfrak{sl}(n,\R)$, $\Vert\xi-\xi_\varepsilon\Vert=\varepsilon\Vert P^{-1}AP\Vert$. This proves :\\

\begin{lem}

\

The set $\Omega$ is dense in $\mathfrak{sl}(n,\R)$.\\
\end{lem}

Show that the set $\Omega$ is an open :\\

We use the implicit function theorem. Let $\xi_0$ a matrix in $\Omega$. The characteristic polynomial $C_{\xi_0}$ of $\xi_0$ has $n$ simple roots. If $C_{\xi_0}(\alpha)=0$, then $C'_{\xi_0}(\alpha)\neq 0$.

If $c_j$ is a real eigenvalue of $\xi_0$, then we consider the map $F_j:\mathfrak{sl}(n,\R)\times\R\longrightarrow\R$ defined by $F_j(\xi,x)=C_\xi(x)$.

If $a_k+ib_k$ is a non real eigenvalue of $\xi_0$, we note $C_\xi(z)=C_\xi(x+iy)=A_\xi(x,y)+i B_\xi(x,y)$, where $A_\xi$ and $B_\xi$ are real. $C_\xi$ is a polynomial in $z$, then
$$
\frac{\partial}{\partial\overline{z}}C_\xi(z)=\frac{\partial}{\partial x}C_\xi(z)+i \frac{\partial}{\partial y}C_\xi(z)=0,
$$
for all $z$. Since $a_k+ib_k$ is a simple root of $C_{\xi_0}$, then 
$$
\frac{\partial}{\partial z}C_{\xi_0}(a_k+ib_k)=\frac{\partial}{\partial x}C_{\xi_0}(a_k+ib_k)-i \frac{\partial}{\partial y}C_{\xi_0}(a_k+ib_k)\neq0.
$$
Therefore, we have either $\frac{\partial}{\partial x}A_{\xi_0}(a_k+ib_k)\neq0$ and $\frac{\partial}{\partial y}B_{\xi_0}(a_k+ib_k)\neq0$ or $\frac{\partial}{\partial y}A_{\xi_0}(a_k+ib_k)\neq0$ and $\frac{\partial}{\partial x}B_{\xi_0}(a_k+ib_k)\neq0$. In all cases,
$$
\frac{D(A_{\xi_0},B_{\xi_0})}{D(x,y)}(a_k,b_k)=\left|\begin{matrix}\partial_x A_{\xi_0}&\partial_yA_{\xi_0}\\ \partial_xB_{\xi_0}&\partial_y B_{\xi_0}\end{matrix}\right|(a_k,b_k)=(\partial_x A_{\xi_0}(a_k,b_k))^2+(\partial_xB_{\xi_0}(a_k,b_k))^2\neq0.
$$
We define $F_j:\mathfrak{sl}(n,\R)\times\R^2\longrightarrow\R^2$ by $F_k(\xi,x,y)=(A_\xi(x+iy),B_\xi(x+iy))$, then  $Jac(F_k)(\xi_0,a_k,b_k)\neq0$.\\

The functions $F_j$, $F_k$ are differentiable, then $F_j(\xi_0,c_j)=0$ and $\frac{\partial F_j}{\partial x}(\xi_0,c_j)=C'_{\xi_0}(c_j)\neq0$. Similarly, $F_k(\xi_0,a_k,b_k)=0$ and $\frac{D F_k}{D(x,y)}(\xi_0,a_k,b_k)\neq0$.\\

So, there exists an open $U_j$ (resp. $U_k$) of $\mathfrak{sl}(n,\R)$, containing $\xi_0$, and there is an open $V_j$ of $\R$ containing $c_j$ (resp. $V_k$ of $\R^2$,  containing $(a_k,b_k)$) and there are maps $f_j:U_j\longrightarrow V_j$ (resp. $f_k:U_k\longrightarrow V_k$) such that 
$$
\left.\aligned
(\xi,x)\in U_j\times V_j\cr
F_j(\xi,x)=0\endaligned\right\}~~\Longleftrightarrow~~\left\{\aligned (\xi,x)\in U_j\times V_j\cr 
x=f_j(\xi)\endaligned\right.
$$
$$
\quad\Big(\text{resp.}~~
\left.\aligned
(\xi,x,y)\in U_k\times V_k\cr
F_k(\xi,x,y)=0\endaligned\right\}~~\Longleftrightarrow~~\left\{\aligned (\xi,x,y)\in U_k\times V_k\cr 
(x,y)=f_k(\xi)\endaligned\right.\Big)
$$

We replace as needed the open $V_r$ by another open small enough such that
$$
V_j\cap \left(\bigcup_{j'\neq j}V_{j'}\right)=\emptyset,\quad
V_k\cap \left((\R\times\{0\})\cup\bigcup_{k'\neq k}V_{k'}\right)=\emptyset.
$$
And we put $U=\bigcap_r U_r$. $U$ is an open containing $\xi_0$ and, for all $\xi$ in $U$, $C_\xi$ vanishes at $n$ distinct points (real or complex), then, $U\subset\Omega$.\\

\begin{lem}

\

The set $\Omega$ is an open in $\mathfrak{sl}(n,\R)$.\\
\end{lem}


\section{$\mathfrak{sl}(n,\R)$ admits an overalgebra almost separating of degree $n$}


\subsection{Separation of orbits of $\Omega$ by invariant functions }


\

This is also well known. Let $\xi$ a $n\times n$ real matrix and $C_\xi$ its characteristic polynomial. On $\C$, we can put $\xi$ in Jordan form. We note $z_1,\dots,z_n$ the diagonal terms of this Jordan form. Then :
$$
\aligned
C_\xi(X)&=(-1)^n\det(\xi-X I)=(X-z_1)\cdots(X-z_n)\\
&=X^n-(\sum_i z_i)X^{n-1}+(\sum_{i<j}z_iz_j)X^{n-2}+\dots+(-1)^nz_1\dots z_n\\
&=X^n-\alpha_{n-1}X^{n-1}+\alpha_{n-2}X^{n-2}+\dots+(-1)^n\alpha_0.
\endaligned
$$
Therefore, using a formula due to Newton ({\sl{cf.}} \cite{W}), we have, for all $k$,
$$
\aligned
(-1)^{k+1}\sum_{i_1<\dots<i_k}z_{i_1}z_{i_2}\dots z_{i_k}&=\sum_jz_j^k-(\sum_{i_1}z_{i_1})(\sum_j z_j^{k-1})+(\sum_{i_1<i_2}z_{i_1}z_{i_2})(\sum_j z_j^{k-2})+\dots\\
&\hskip 1cm+\dots+(-1)^{k-1}(\sum_{i_1<\dots<i_{k-1}}z_{i_1}\dots z_{i_{k-1}})(\sum_j z_j)\\
\endaligned
$$
or
$$
(-1)^{k+1}\alpha_{n-k}=Tr(\xi^k)-\alpha_{n-1}Tr(\xi)+\alpha_{n-2}Tr(\xi^{k-1})+\dots+(-1)^{k-1}\alpha_{k-1}Tr(\xi).
$$
This formula allows to express all $\alpha_k$ as functions of the numbers $Tr(\xi^j)$, and conversely, to express all $Tr(\xi^k)$ as functions of the numbers $\alpha_j$.\\

Finally, we deduce that:

two matrices $\xi$ and $\xi'$ satisfying $C_{\xi}=C_{\xi'}$ if and only if $Tr(\xi^k)=Tr(\xi'^k)$ for all $k=1,\dots,n$.\\

\begin{prop}

\

We keep all previous notations, in particular, $\Omega=\cup_{r+2s=n}\Omega_{r,s}$ is an open, dense and invariant subset of $\mathfrak{sl}(n,\R)$. The orbits of $\Omega$ will be called generic orbits. Let $\xi_0\in\Omega_{r,s}$.
\begin{itemize}
\item[1.] If $r>0$, $\{\xi,$ such that $T_k(\xi)=T_k(\xi_0)$ for all $k=2,\dots,n~\}$ is exactly the adjoint orbit $G\cdot\xi_0$ of $\xi_0$,\\
\item[2.] If $r=0$, $\{\xi,$ such that $T_k(\xi)=T_k(\xi_0)$ for all $k=2,\dots,n~\}$ is the union of two adjoint orbits $G\cdot\xi_0\sqcup G\cdot\xi_1$.\\
\end{itemize}
\end{prop}

We say that the invariant functions $T_k$ separate almost the (co)adjoint generic orbits of $\mathfrak{sl}(n,\R)$.\\


\subsection{Convex hull of orbits of $\Omega$}

\

\

For $n=2$, the convex hull of the orbits of $\Omega$ are well known (cf.\cite{ASZ}). We deduce that, for $n=2$ :
$$
\text{Conv}(G\cdot D(-c,c))=\mathfrak{sl}(2,\R)\quad (c>0),\quad \text{ and }
$$
$$
\text{Conv}(G\cdot D^+(ib)\cup G\cdot D^-(ib))=\mathfrak{sl}(2,\R)\quad (b>0).
$$

For $n=3$, we deduce that $\Omega\subset\text{Conv}(G\cdot D(c_1,c_2,c_3))$. Indeed, if $c'_1<c'_2<c'_3$ such that $c'_1+c'_2+c'_3=0$, then either $c'_1\neq-2c_3$, or $c'_2\neq-2c_3$. Suppose $c'_1\neq-2c_3$, the other case is trained the same by exchanging the induces 1 and 2. Let $c''_1=c'_1-\frac{c_1+c_2}{2}$ and $c''_2=-c'_1-c_3$. We write:
$$
\left(\begin{matrix}c_1&\\&c_2\end{matrix}\right)=\left(\begin{matrix}\frac{1}{2}(c_1-c_2)&\\&\frac{1}{2}(c_2-c_1)\end{matrix}\right)+\frac{c_1+c_2}{2} \left(\begin{matrix}1&\\&1\end{matrix}\right),
$$
then, there exist $t$ in $[0,1]$ and $g\in SL(2,\R)$ such that :
$$
\left(\begin{matrix}c'_1&\\&c''_2\end{matrix}\right)=\left(\begin{matrix}c''_1&\\&-c''_1\end{matrix}\right)+\frac{c_1+c_2}{2} \left(\begin{matrix}1&\\&1\end{matrix}\right)=t\left(\begin{matrix}c_1&\\&c_2\end{matrix}\right)+(1-t)g\left(\begin{matrix}c_1&\\&c_2\end{matrix}\right)g^{-1}.
$$
We deduce that the convex hull of $G\cdot D(c_1,c_2,c_3)$ contains $\left(\begin{matrix}c'_1&&\\&c''_2&\\&&c_3\end{matrix}\right)$ with $c''_2\neq c_3$.  By the same argument, but with induces 2 and 3, we show that this convex hull contains $D(c'_1,c'_2,c'_3)$. Let now $a'=-\frac{1}{2}c'_1$, and $b'>0$, then :
$$
\left(\begin{matrix}c'_2&\\&c'_3\end{matrix}\right)=\left(\begin{matrix}\frac{1}{2}(c'_2-c'_3)&\\&\frac{1}{2}(c'_3-c'_2)\end{matrix}\right)+a' \left(\begin{matrix}1&\\&1\end{matrix}\right),
$$
the convex hull of $G\cdot D(c_1,c_2,c_3)$ contains also 
$$
D(c'_1,a'+ib')=\left(\begin{matrix}c'_1&&\\&0&b'\\&-b'&0\end{matrix}\right)+a'\left(\begin{matrix}0&&\\&1&\\&&1\end{matrix}\right).
$$

On the other hand, we saw that $\left(\begin{matrix}c&&\\&a&-b\\&b&a\end{matrix}\right)$ belongs to $G\cdot D(c,a+ib)$. Therefore, if $a\neq 0$ then $\text{Conv}(G\cdot D(c,a+ib))$ contains the matrix $\left(\begin{matrix}c&&\\&a&\\&&a\end{matrix}\right)$, with $a\neq c$. So, by our first argument, $\text{Conv}(G\cdot D(c,a+ib))$ contains the matrix $\left(\begin{matrix}c'_1&&\\&c'_2&\\&&a\end{matrix}\right)$ with $c'_1\neq c'_2\neq a\neq c'_1$.Therefore, by the above, $\text{Conv}(G\cdot D(c,a+ib))$ contains all $\Omega$. If $a=0$, $\text{Conv}(G\cdot D(0,ib))$ contains the matrices $\left(\begin{matrix}0&\\&D^+(ib)\end{matrix}\right)$ and $\left(\begin{matrix}0&\\&D^-(ib)\end{matrix}\right)$. So, $\text{Conv}(G\cdot D(0,ib))$ contains the matrix $\left(\begin{matrix}0&&\\&-1&\\&&1\end{matrix}\right)$. Finally, $\text{Conv}(G\cdot D(0,ib))$ contains all $\Omega$.

We have proved:

\begin{lem}

\

If $n=3$ and $\xi\in\Omega$, then $\Omega\subset\text{\rm Conv}(G\cdot \xi)$.\\
\end{lem}

If $n=4$, as above, $\Omega\subset\text{Conv}(G\cdot D(c_1,\dots,c_4))$. We deduce by using the lemma that for all $c'_2,\dots,c'_4$,
$$
D(c_1,c'_2,c'_3,c'_4)\in\text{Conv}(G\cdot D(c_1,c_2,a+ib))
$$
and also $\Omega\subset\text{Conv}(G\cdot D(c_1,c_2,a+ib))$. It remains the cases $D(a_1+ib_1,a_2+ib_2)$, $a_1\neq0$ and $D(ib_1,ib_2)$.
In the first case, we saw that
$$
\left(\begin{matrix}a_1&-b_1&&\\b_1&a_1&&\\&&a_2&-b_2\\&&b_2&a_2\end{matrix}\right)\in G\cdot D(a_1+ib_1,a_2+ib_2),
$$
then
$$
\left(\begin{matrix}a_1&&&\\&a_1&&\\&&a_2&\\&&&a_2\end{matrix}\right)\in\text{Conv}( G\cdot D(a_1+ib_1,a_2+ib_2))\qquad\text{ and } a_1\neq a_2.
$$
By applying the calculation for $n=2$, we deduce that $D(a_1,x,y,a_4)$ belongs to $\text{Conv}( G\cdot D(a_1+ib_1,a_2+ib_2))$, for all $x$ and $y$ such that $a_1+x+y+a_4=0$. Therefore, $\Omega\subset \text{Conv}( G\cdot D(a_1+ib_1,a_2+ib_2))$.

For the latter case, we saw that, in$\mathfrak{sl}(2,\R)$, the adjoint orbit of $D(ib)$ is the set of matrices $\left(\begin{matrix}x&y+z\\y-z&-x\end{matrix}\right)$ with $z^2-x^2-y^2=b^2$ and $z>0$. Then, $G\cdot D(ib_1,ib_2)$ contains a matrix as follows:
$$
\left(\begin{matrix}x&z&&\\-z&-x&&\\&&0&b_2\\&&-b_2&0\end{matrix}\right),\qquad\text{ with}\quad 0<b_1<z.
$$
Combining this matrix with $\left(\begin{matrix}0&1&&\\1&0&&\\&&0&1\\&&1&0\end{matrix}\right)$, we obtain :
$$
\xi=\left(\begin{matrix}-x&-z&&\\z&x&&\\&&0&-b_2\\&&b_2&0\end{matrix}\right)\in G\cdot D(ib_1,ib_2).
$$
If $t=\frac{b_1}{z+b_1}$, the matrix $t\xi+(1-t)D(ib_1,ib_2)$ is $D(-tx,tx,i(1-2t)b_2)\in\text{Conv}(G\cdot D(ib_1,ib_2))$. Or\\

\begin{lem}

\

If $n=4$ and $\xi\in\Omega$, then $\Omega\subset\text{\rm Conv}(G\cdot \xi)$.\\
\end{lem}

\

\begin{prop}

\

For all $n>2$, the convex hull of the adjoint orbit of a point $\xi$ in $\Omega$ contains $\Omega$:
$$
\Omega\subset\text{\rm Conv}(G\cdot \xi).
$$
This convex hull is dense in $\mathfrak{sl}(n,\R)$.\\
\end{prop}

\begin{proof}

\

By induction on $n>4$, suppose that, for all $2<p<n$, this property is true. We consider $D(c_j,a_k+ib_k)\in\Sigma_{r,s}$, with $r+2s=n$.\\

If $r\geq 2$, then $n-2>2$, and we write:
$$
D(c_j,a_k+ib_k)=\left(\begin{matrix} D(c'_1,c'_2)&\\&D(c'_3,\dots,c'_r,a'_k+ib'_k)\end{matrix}\right)+\left(\begin{matrix} \frac{c_1+c_2}{2}I_2&\\&-\frac{c_1+c_2}{n-2}I_{n-2}\end{matrix}\right),
$$
Using the fact that $\mathfrak{sl}(2,\R)=\text{Conv}(G\cdot D(c'_1,c'_2))$ and by the induction hypothesis for $n-2$, we get $\Omega\subset \text{Conv}(G\cdot D(c_j,a_k+ib_k))$.

If $r=1$, we decompose $D(c_1,a_k+ib_k)$ as :
$D(c_1,a_k+ib_k)=$\\
$$
\left(\begin{matrix} D(c'_1,a'_1+ib_1)&\\&D(a'_2+ib'_2,\dots,a'_n+ib'_n)\end{matrix}\right)+\left(\begin{matrix} \frac{c_1+2a_1}{3}I_3&\\&-\frac{c_1+2a_1}{n-3}I_{n-3}\end{matrix}\right).
$$
Then, a matrix $D(c''_1,c''_2,c''_3,a_2+ib_2,\dots,a_n+ib_n)$ belongs to $\text{Conv}(G\cdot D(c_1,a_k+ib_k))$. Therefore, the first case applies, and we still get the result.

If $r=0$, then $s>2$, we decompose $D(a_k+ib_k)$ as :
$D(a_k+ib_k)=$\\
$$
\left(\begin{matrix} D(a'_1+ib_1, a'_2+ib'_2)&\\&D(a'_3+ib'_3,\dots,a'_n+ib'_n)\end{matrix}\right)+\left(\begin{matrix} \frac{2a_1+2a_2}{4}I_4&\\&-\frac{2a_1+2a_2}{n-4}I_{n-4}\end{matrix}\right).
$$
Then a matrix $D(c''_1,\dots,c''_4,a_3+ib_3,\dots,a_n+ib_n)$ belongs to $\text{Conv}(G\cdot D(a_k+ib_k))$. So, the first case applies, and this completes the proof of our proposition.\\

\end{proof}

\begin{cor}

$\mathfrak{sl}(n,\R)$ admits an overalgebra almost separating of degree $n$.\\
\end{cor}

\begin{proof}

\

$\mathfrak g=\mathfrak{sl}(n,\R)$ admits an overalgebra of degree $n$, given by :
$$
\aligned
\mathfrak{g}^+&=\mathfrak{sl}(n,\R)\times \R^{n-1}\cr
\Phi:\mathfrak g^\star&\longrightarrow \mathfrak{g}^{+\star},\qquad \Phi(X)=(X,\phi(X))=(X,T_2(X),T_3(X),\dots,T_n(X)).
\endaligned
$$
Indeed, $\phi$ is polynomial, with degree $n$.\\
Moreover, we have for all $\xi$ in $\Omega$,
$$\aligned
\overline{Conv}\left(\Phi(Coad(SL(n,\R))\xi)\right)&=\overline{Conv}(Coad(SL(n,\R))\xi)\times (T_2(\xi),\dots,T_n(\xi))\\
&=\mathfrak{sl}(n,\R)^\star\times (T_2(\xi),\dots,T_n(\xi)).
\endaligned
$$
Then $\Phi(\xi')$ belongs to this set if and only if $T_k(\xi')=T_k(\xi)$ for all $k$, if and only if $C_{\xi'}=C_\xi$.

We saw that if $n$ is odd, $\{\xi',~~\text{ such that } C_{\xi'}=C_\xi\}$ is exactly the orbit $Coad(SL(n,\R))\xi$ and, if $n$ is even, $\{\xi',~~\text{ such that } C_{\xi'}=C_\xi\}$ is either the orbit $Coad(SL(n,\R))\xi$, or, if $C_\xi$ has only non real roots, the set $\{\xi',~~\text{ such that } C_{\xi'}=C_\xi\}$ is the union of two disjoint orbits. This proves that $(\mathfrak g^+,\phi)$ is an overalgebra almost separating of degree $n$ of $\mathfrak{sl}(n,\R)$.

\end{proof}

\


\section{Overalgebra almost separating of degree $p$ of a Lie algebra $\mathfrak g$}


\

\begin{defn}$($Semi direct product$)$
\

Let $G$ be a real Lie group, $V$ a finite dimensional vector space and $(\pi, V)$ a linear representation of $G$. Denote by $G^+=G'\rtimes V$ the Lie group whose set $G\times V$ and low:
$$
(g,v).(g',v')=(gg',v+\pi(g)v').
$$

Its Lie algebra is $\mathfrak g^+=\mathfrak g'\rtimes V$, whose space $\mathfrak g\oplus V$ and bracket :
$$
[(X,u),(X',u')]=([X,X'],\pi'(X)u'-\pi'(X')u).
$$
$($ $\pi'$ is the derivative of $\pi$, $\pi'$ is the representation of $\mathfrak g$ in $V$$)$.
\end{defn}

The exponential map is
$$
\exp(X,u)=(\exp X,\frac{e^{\pi'(X)}-I}{\pi'(X)}u).
$$

We also define the linear map $\psi_u:\mathfrak g~\longrightarrow~V$, by $\psi_u(X)=\pi'(X)u$, for all $u\in V$. Then, the coadjoint action is realized in ${\mathfrak g^+}^\star=\mathfrak g^\star\times V^\star$ and defined by:
$$
Coad'(X,u)(\xi,f)=(Coad'(X)\xi+^t\psi_u(f),-^t\pi'(X)f).
$$

The group action is
$$
Coad(g,v)(\xi,f)=(Coad(g)\xi+^t\psi_v(^t\pi(g^{-1})f),^t\pi(g^{-1})f).
$$
Denote by $\pi^\star(g)=^t\pi(g^{-1})$. 

Let $\Phi:\mathfrak g^\star~\longrightarrow~{\mathfrak g^+}^\star$ be a map non necessarily linear. We assume that $p\circ \Phi=id$, then $\Phi$ is written $\Phi(\xi)=(\xi,\phi(\xi))$. $\phi$ is not necessarily linear.\\

Assume that $\Phi(Coad(G)\xi)=Coad(G^+)\Phi(\xi)$, then : for all $g$ in $G$ and all $v$ in $V$, there exists $g'\in G$ ($g'=g'_{g,v,\xi}$) such that
$$\left\{\aligned
\pi^\star(g)\phi(\xi)&=\phi(Coad(g')\xi),\\
Coad(g)(\xi)+(^t\psi_v\circ \pi^\star(g))\phi(\xi)&=Coad(g')\xi=Coad(g)(\xi)+^t\psi_v\circ \phi(Coad(g')\xi).\\
\endaligned\right.
$$

In particular, if $X$ is in $\mathfrak g$, then $\pi^\star(\exp(tX))(\phi(\xi))=\phi(Coad(g'_t)\xi)$. The continuous curve $t\mapsto\pi^\star(\exp(tX))(\phi(\xi))$ is drawn on the surface $\mathcal C=\phi(Coad(G)\xi)$, its derivative at 0 is the vector ${\pi^\star}'(X)\phi(\xi)$. This vector belongs to the tangent space
$$
T_{\phi(\xi)}(\mathcal C)=\phi'(\phi(\xi))(T_\xi(Coad(G)\xi))=\phi'(\phi(\xi))(Coad(\mathfrak g)\xi).
$$

We have also, for the same $g$ in $G$, $v$ in $V$, and $g'=g'_{g,v,\xi}\in G$,
$$
Coad(g)(\xi)=(I-^t\psi_v\circ\phi)(Coad(g')\xi).
$$
We deduce that, if $v=0$, then $Coad(g)(\xi)=Coad(g'_{g,0,\xi})(\xi)$ and therefore $\pi^\star(g)\phi(\xi)=\phi(Coad(g)\xi)$. So:

\begin{lem}
\

$\phi$ is an intertwining (non linear) between the coadjoint representation and the representation $\pi^\star$.\\
\end{lem}

If $\phi$ is polynomial of degree $p$, then $\phi$ is written :
$$
\phi(\xi)=\phi_1(\xi)+\phi_2(\xi)+\dots+\phi_p(\xi),
$$
with $\phi_k$ homogeneous of degree $k$.

Since $\phi$ is an intertwining, then $\phi\circ Ad_g=\pi^\star(g)\circ \phi$, and for all $k$, $\phi_k \circ Ad_g=\pi^\star(g)\circ \phi_k$, i.e each $\phi_k$ is an intertwining.\\

On the other hand, for each $k$, $\phi_k(\xi)$ can be written 
$$
\phi_k(\xi)=b_k(\underbrace{\xi\cdot \ldots \cdot\xi}_k),
$$
where $b_k$ is a linear map from $S^k(\mathfrak g^\star)$ in $V^\star$. The map $b_k$ is also an intertwining, since the action $Coad^k$ of $G$ on $S^k(\mathfrak g^\star)$ is such that :
$$
\phi_k(Coad(g)\xi)=b_k(Coad(g)\xi\cdot \ldots\cdot Coad(g)\xi)=(b_k\circ Coad^k(g))(\xi\cdot\ldots\cdot\xi).
$$

Put then:
$$
S_p(\mathfrak g^\star)=\mathfrak g^\star\oplus S^2(\mathfrak g^\star)\oplus\dots\oplus S^p(\mathfrak g^\star),
$$
and 
$$
b:~S_p(\mathfrak g^\star)~\longrightarrow ~~ V^\star,~~b(v_1+v_2+\dots+v_p)=b_1(v_1)+\dots+b_p(v_p).
$$
Let $U=b(S_p(\mathfrak g^\star))$. $U$ is a submodule of $V^\star$, isomorphic to the quotient module $S_p(\mathfrak g^\star)/\ker(b)$. Put then $W=V/U^\perp$. $W$ is a quotient module of the module $V$ such that $W^\star\simeq U$ ( and then $W\simeq U^\star$).\\

\begin{lem}

\

If $(\mathfrak g\rtimes V,\phi)$ is an overalgebra almost separating of $\mathfrak g$, then $(\mathfrak g\rtimes W,\tilde{\phi})$, where 
$$
\tilde{\phi}(\xi)=b(\xi+\xi\cdot \xi+\dots+\xi\cdot\ldots\cdot\xi)
$$
is also an overalgebra almost separating of $\mathfrak g$.\\
\end{lem}

\begin{proof}

\

In the statement of this lemma, we identify $W^\star$ with the submodule $U$ of $V^\star$. With this identification, if $\iota$ is the canonical injection of $U$ in $V^\star$, then $\phi(\xi)=\iota\circ\tilde{\phi}(\xi)$. The application $\Phi$ becomes $\tilde{\Phi}(\xi)=(\xi,\iota\circ\phi(\xi))=(j\circ\Phi)(\xi)$ if $j(\xi,v)=(\xi,\iota(v))$. Therefore
$$
\overline{Conv}\left(\tilde{\Phi} (Coad G~\xi)\right)=j\left(\overline{Conv}\left(\Phi (Coad G~\xi)\right)\right),
$$
and $\overline{Conv}\left(\tilde{\Phi} (Coad G~\xi)\right)=\overline{Conv}\left(\tilde{\Phi} (Coad G~\xi')\right)$ if and only if $\overline{Conv}\left(\Phi (Coad G~\xi)\right)=\overline{Conv}\left(\Phi (Coad G~\xi')\right)$. We deduce that $(\mathfrak g\rtimes W,\tilde{\phi})$ or, if we prefer, $(\mathfrak g\rtimes \left(S_p(\mathfrak g^\star)/\ker b\right)^\star,\tilde{\phi})$ is an overalgebra almost separating of $\mathfrak g$.\\

\end{proof}

If $\mathfrak g$ is semi-simple and deployed, then all its representations are completely reducible. Therefore $W^\star=S_p(\mathfrak g^\star)/\ker b$ is isomorphic to a submodule of $S_p(\mathfrak g^\star)=S_p(\mathfrak g)$. In this case, $W$ is isomorphic to a submodule of $\left(S_p(\mathfrak g)\right)^\star$. So, we consider the application $\phi$ with values in $S_p(\mathfrak g)$, and $\phi$ becomes :
$$
\phi(\xi)=b_1(\xi)+\dots+b_p(\xi\cdot\ldots\cdot\xi).
$$
The application $b$ becomes an intertwining of modules $S_p(\mathfrak g)$.

\begin{cor}

\

If $\mathfrak g$ is a deployed and semi-simple Lie algebra, admitting an overalgebra almost separating of degree $p$, and $\tau$ a natural application from $\mathfrak g=\mathfrak g^\star$ to $S_p(\mathfrak g)$ defined by : $\tau(\xi)=\xi+\xi\cdot\xi+ \dots +\xi\cdot\ldots\cdot\xi$, then there exists an intertwining $b$ of $S_p(\mathfrak g)$ such that $(\mathfrak g\rtimes \left(S_p(\mathfrak g)\right)^\star,b\circ\tau)$ is an overalgebra almost separating of $\mathfrak g$.

\end{cor}

\section{The case $\mathfrak g=\mathfrak{sl}(n,\R)$ and $p=2$}


\


\subsection{The case $\mathfrak g=\mathfrak{sl}(n,\R)$}


\

We suppose now $\mathfrak g=\mathfrak{sl}(n,\R)$. Recall the usual notations ({\sl{cf.}} \cite{FH}).

$\mathfrak{sl}(n,\R)$ is a real simple algebra. A Cartan subalgebra $\mathfrak h$ of $\mathfrak{sl}(n,\R)$, of dimension $n-1$, is given by the set of diagonal matrices $\xi=D(c_1,\dots,c_n)$. With this Cartan algebra, $\mathfrak{sl}(n,\R)$ is deployed. For $H\in\mathfrak h$, we note $L_i(H)=c_i$, with $L_1+\dots+L_n=0$. We choose the usual system of simple roots, i.e the forms $\alpha_i=L_i-L_{i+1}$ ($1\leq i\leq n-1$). The system of positive roots is the set of forms $L_i-L_j$, with $i<j$. If $e_{ij}=\left(x_{rs}\right)$ is the $n\times n$ matrix such that $x_{rs}=\delta_{ri}\delta_{sj}$, then for all $H$ in $\mathfrak h$,
$$
\text{ad}(H)e_{ij}=\left\{\aligned&(\alpha_i+\dots+\alpha_{j-1})(H)e_{ij},&\quad\text{ if }i<j\\
&-(\alpha_j+\dots+\alpha_{i-1})(H)e_{ij},&\quad\text{ if }i>j.\endaligned\right.
$$
The fundamental weights are $\omega_k=L_1+\dots+L_k$ ($1\leq k\leq n-1$) and the simple modules are exactly the modules noted $\Gamma_{a_1\dots a_{n-1}}$ of highest weight $a_1\omega_1+\dots+a_{n-1}\omega_{n-1}$, with $a_k$ integer. Moreover, the dual of $\Gamma_{a_1,\dots,a_{n-1}}$ is the module $\Gamma_{a_{n-1},\dots,a_1}$. ({\sl{cf.}} \cite{FH}).\\

Let $X\mapsto X^s$ be the symmetry operation relative to the second diagonal given by:

if $X$ is the matrix $(x_{ij})$, then $X^s$ is the matrix $(x^s_{ij})$ with :
$$
x^s_{ij}=x_{(n+1-j)(n+1-i)},
$$
The operation $s$ leaves the Cartan subalgebra invariant. For all weight $\omega$, put $\omega^s(H)=\omega(H^s)$. In particular, $L_i^s=L_{n+1-i}$ and $\omega_j^s=-\omega_{n-j}.$\\
Moreover, $s$ permutes the radiciel spaces, since $e_{ij}^s=e_{(n+1-j)(n+1-i)}$, or, if $H$ belongs to $\mathfrak h$ and $i<j$,
$$
[H,e_{ij}^s]=-(L_i^s-L_j^s)(H)e_{ij}^s.
$$

We consider now the module $S^k(\mathfrak{sl}(n,\R))$ i.e the space of the sums $\sum_{i_1<\dots<i_k} \lambda_{i_1,\dots,i_k}X_{i_1}\cdot\ldots\cdot X_{i_k}$ on which $\mathfrak{sl}(n,\R)$ acts by the adjoint action $ad$ defined by:
$$
ad_X(X_{i_1}\cdot\ldots\cdot X_{i_k})=\sum_{r=1}^k X_{i_1}\cdot\ldots\cdot[X,X_{i_r}]\cdot\ldots\cdot X_{i_k}.
$$
\begin{lem}
\

The space $S^k(\mathfrak{sl}(n,\R))$ is self-dual i.e $\left(S^k(\mathfrak{sl}(n,\R))\right)^\star=S^k(\mathfrak{sl}(n,\R))$.
\end{lem}
\begin{proof}
\

Suppose that the module $\Gamma_{a_1,\dots,a_{n-1}}$ appears in $S^k(\mathfrak{sl}(n,\R))$. Then, there is a non zero vector $v_{a_1,\dots,a_{n-1}}$ such that , for all $H$ in $\mathfrak h$, and for all $i<j$,
$$
ad_H(v_{a_1\dots a_{n-1}})=(a_1\omega_1+\dots+a_{n-1}\omega_{n-1})(H)v_{a_1\dots a_{n-1}},\quad\quad ad_{e_{ij}}(v_{a_1\dots a_{n-1}})=0.
$$
If $v=\sum_{i_1<\dots<i_k} \lambda_{i_1,\dots,i_k}X_{i_1}\cdot\ldots\cdot X_{i_k}$ is a vector of $S^k(\mathfrak{sl}(n,\R))$, then $v^s=\sum_{i_1<\dots<i_k} \lambda_{i_1,\dots,i_k}X^s_{i_1}\cdot\ldots\cdot X^s_{i_k}$. Moreover, the map $v\mapsto v^s$ is an involutive bijection : $(v^s)^s=v$.

The vector $v^s_{a_1\dots a_{n-1}}$ is not zero and, for all $H$ in $\mathfrak h$, and all $i<j$,
$$\aligned
ad_H(v^s_{a_1\dots a_{n-1}})&=-(a_1\omega_1+\dots+a_{n-1}\omega_{n-1})^s(H)v^s_{a_1\dots a_{n-1}}\\
&=(a_{n-1}\omega_1+\dots+a_1\omega_{n-1})(H)v^s_{a_1\dots a_{n-1}},\\ 
ad_{e_{ij}}(v^s_{a_1\dots a_{n-1}})&=0.\\
\endaligned
$$

In other words,
$$
\left(\Gamma_{a_1,\dots,a_{n-1}}\right)^s=\Gamma_{a_{n-1},\dots,a_1}\simeq\left(\Gamma_{a_1,\dots,a_{n-1}}\right)^\star.
$$
\end{proof}

\begin{cor}

\

If $\mathfrak{sl}(n,\R)$ admits an overalgebra almost separating of degree $p$, if $\tau$ is the natural application of $\mathfrak{sl}(n,\R)^\star$ in $S_p(\mathfrak{sl}(n,\R))=\left(S_p(\mathfrak{sl}(n,\R))\right)^\star$ given by : $\tau(\xi)=\xi+\xi\cdot\xi+ \dots +\xi\cdot\ldots\cdot\xi$, then there exist an intertwining $b_k$ of $S^k(\mathfrak{sl}(n,\R))$ ($k=1,\dots,p$) such that $(\mathfrak{sl}(n,\R)\rtimes S_p(\mathfrak{sl}(n,\R)),\sum_kb_k\circ\tau)$ is an overalgebra almost separating of the Lie algebra $\mathfrak{sl}(n,\R)$.\\

\end{cor}


\subsection{The case $p=2$}


\

For $k=1$ and $k=2$, looking for the intertwining between $S^k(\mathfrak{sl}(n,\R)^\star)$ and $\left(S^k(\mathfrak{sl}(n,\R))\right)^\star$.\\
For $k=1$, the space of these intertwining is one dimensional and generated by $P_0$ defined by :
$$
\langle P_0(\xi),X\rangle=Tr(\xi X).
$$
Therefore, any intertwining $b_1$ is written $b_1=a_0P_0$, with $a_0$ real. For k=2:

\subsubsection{Decomposition of $S^2(\mathfrak g)$}

\

The module $S^2(\mathfrak g)$ is the sum of three or four irreducible modules, all of different types. Recall the usual notations ({\sl{cf.}} \cite{FH}).
\begin{itemize}
\item The highest weight vector in $S^2(\mathfrak{sl}(n,\R))$ is 
$$
v_{20\dots 02}=e_{1n}.e_{1n}.
$$ 
The weight of this vector is $2\omega_1+2\omega_n=2L_1-2L_n$. Then, we deduce the existence of a simple module $\Gamma_{20\dots 02}$ of dimension ({\sl{cf.}} \cite{FH}):
$$
\dim \Gamma_{20\dots 02}=\displaystyle\prod_{i=2}^{n-1}\frac{2+n-i}{n-i}.\displaystyle\prod_{j=2}^{n-1}\frac{2+j-1}{j-1}.\frac{4+n-1}{n-1}
=\frac{n^2(n-1)(n+3)}{4}.
$$ 
\item Among the weight vectors of weight $\omega_2+\omega_{n-2}=L_1-L_n+L_2-L_{n-1}$, and if $n>3$, there is one that is annulled by the action of $e_{i(i+1)}$, $1\leq i\leq (n-1)$. This weight vector is :
$$
v_{010\dots 010}=e_{2n}.e_{1(n-1)}-e_{2(n-1)}.e_{1n}.
$$
We deduce then the existence of a simple module $\Gamma_{010\dots 010}$ of dimension:
$\dim \Gamma_{010\dots 010}=$\\
$$
\begin{array}{lll}
&&\displaystyle\prod_{j=3}^{n-2}\frac{1+(j-1)}{j-1}.\displaystyle\prod_{j=3}^{n-2}\frac{1+(j-2)}{j-2}.\displaystyle\prod_{i=3}^{n-2}\frac{1+(n-1)-i}{n-1-i}.\displaystyle\prod_{i=3}^{n-2}\frac{1+n-i}{n-i}.\frac{n^2(n+1)}{4(n-2)^2(n-3)}\\
&&=\frac{n^2(n+1)(n-3)}{4}
\end{array}
$$
If $n=3$, $e_{22}$ is not in $\mathfrak{sl}(3)$, and this dimension is 0. This sub module does not appear. 
\item Among the weight vectors of weight $\omega_1+\omega_{n-1}=L_1-L_n$, there is one that is annulled by the action of $e_{i(i+1)}$, $1\leq i\leq (n-1)$. This weight vector is:
$$
v_{10\dots 01}=\displaystyle\sum_{i=1}^{n}e_{1i}.e_{in}-\frac{2}{n}\displaystyle\sum_{j=1}^{n}e_{jj}.e_{1n}.
$$
Then, we deduce the existence of a simple module $\Gamma_{10\dots 01}$ of dimension: 
$$
\dim \Gamma_{10\dots 01}=\displaystyle\prod_{i=2}^{n-1}\frac{1+n-i}{n-i}.\displaystyle\prod_{j=2}^{n-1}\frac{1+j-1}{j-1}.\frac{2+n-1}{n-1}=n^2-1.
$$
\item Among the weight vectors of weight $0$, there is one that is annulled by the action of $e_{i(i+1)}$, $1\leq i\leq (n-1)$. This weight vector is:
$$
v_{00\dots 00}=2n\displaystyle\sum_{1\leq i<j\leq n}e_{ij}.e_{ji}+\displaystyle\sum_{1\leq i<j\leq n}(e_{ii}-e_{jj}).(e_{ii}-e_{jj})
$$
We deduce the existence of a trivial simple module $\Gamma_{00\dots 00}$ of dimension $1$.
\end{itemize}
Therefore:
$$
S^2(\mathfrak{sl}(n,\R))\cong\left\{\aligned
&\Gamma_{20\dots 02}\oplus\Gamma_{10\dots 01} \oplus \Gamma_{010\dots 010}\oplus \Gamma_{0\dots 0},\quad&\text{if }n>3\\
&\\
&\Gamma_{22}\oplus\Gamma_{11} \oplus\Gamma_{00},\quad&\text{if }n=3.
\endaligned\right.
$$

since the dimensions, $\frac{n^2(n^2-1)}{2}=\frac{n^2(n-1)(n+3)}{4}+n^2-1+\frac{n^2(n+1)(n-3)}{4}+1$.

\subsubsection{Intertwining of $S^2(\mathfrak{sl}(n,\R))$}

\ 

Let $P_1$, $P_2$, $P_3$ and $P_4$ the intertwining defined from $S^2(\mathfrak{sl}(n,\R))$ in $(S^2(\mathfrak{sl}(n,\R)))^\star$, such that, for all $\xi$, $\eta\in \mathfrak{sl}(n,\R)$ and $X,~ Y\in \mathfrak{sl}(n,\R)$ :

\begin{itemize}
\item $\left\langle P_1(\xi.\eta), X.Y \right\rangle=Tr(\xi X\eta Y)+Tr(\xi Y\eta X)$,
\item $\left\langle P_2(\xi.\eta), X.Y \right\rangle=Tr(\xi X)Tr(\eta Y)+Tr(\xi Y)Tr(\eta X)$,
\item $\left\langle P_3(\xi.\eta), X.Y \right\rangle=Tr(\xi\eta XY)+Tr(\xi\eta YX)+Tr(\eta\xi XY)+Tr(\eta\xi YX)$,
\item $\left\langle P_4(\xi.\eta), X.Y \right\rangle=Tr(\xi\eta)Tr(XY)$.
\end{itemize} 

In particular, we have :
\begin{itemize}
\item $P_2(v_{20\dots02})=P_3(v_{20\dots02})=P_4(v_{20\dots02})=0$, and $\left\langle P_1(v_{20\dots02}), e_{n1}.e_{n1} \right\rangle=1\neq 0$,
\item $P_3(v_{010\dots010})=P_4(v_{010\dots010})=0$, and $\left\langle P_2(v_{010\dots010}), e_{n2}.e_{(n-1)1} \right\rangle=4\neq 0$,
\item $P_4(v_{10\dots01})=0$, and $\left\langle P_3(v_{10\dots01}), e_{n2}.e_{21}\right\rangle\neq 0$,
\item $P_4(v_{00\dots00})\neq0$.
\end{itemize}

Thus, if $n>3$, $P_1$, $P_2$, $P_3$ and $P_4$ are independent and, since the dimension of the space of intertwining of $S^2(\mathfrak{sl}(n,\R))$ is 4, then any intertwining $b_2$ is written :
$$
b_2=a_1P_1+a_2P_2+a_3P_3+a_4P_4, \quad\text{where $a_i$ are real constants}. 
$$

If $n=3$, $P_1$, $P_3$ and $P_4$ are independent and, since the dimension of the space of intertwining is three, then we can write :
$$
b_2=a_1P_1+a_3P_3+a_4P_4, \quad\text{where $a_i$ are real constants}. 
$$

\begin{rema}
First, recall that for $\mathfrak{gl}(n,\R)$, the forms $(A_1,\dots,A_m)\mapsto Tr(A_{i_1}\dots A_{i_k})$ are the only invariant functions which generate $K[End(V)^m]^{GL(V)}$ (see [H-C]). 

Remark that there are $24$ possible products of $4$ matrices, depending on the position of the matrix in the product. If we take the trace of these products, then there are only $6$ distinct forms, since, for all $A_1, A_2, A_3, A_4\in\mathfrak{sl}(n,\R)$:
$$
Tr(A_1A_2A_3A_4)=Tr(A_2A_3A_4A_1)=Tr(A_3A_4A_1A_2)=Tr(A_4A_1A_2A_3).
$$
Since we are looking here to build symmetric forms in $\xi$, $\eta$ and $X$, $Y$, there are only $4$ symmetric forms obtained as product of traces of product matrices. These forms are the $4$ forms described above.\\
\end{rema}


\subsection{$\mathfrak{sl}(n,\R)$ does not admit an overalgebra almost separating of degree 2}


\

We have seen if $\mathfrak{sl}(n,\R)$ admits an overalgebra almost separating of degree $2$, then $\mathfrak{sl}(n,\R)$  admits an overalgebra of the form
$$
\mathfrak{G}=(\mathfrak{sl}(n,\R)\rtimes S_2(\mathfrak{sl}(n,\R)),(\phi:\xi\mapsto b_1(\xi)+b_2(\xi\cdot\xi))),
$$
with $b_1=a_0P_0$ and $b_2=a_1P_1+\dots+a_4P_4$.\\

We assume that a such overalgebra almost separating $\mathfrak G$ exists.\\

The generic orbits $SL(n,\R)\cdot\xi$ are the orbits of the points $\xi$ of $\Omega$. Recall that :
$$
SL(n,\R)\cdot\xi\subset\{\xi'\in\mathfrak{sl}(n,\R),~~C_{\xi'}=C_\xi\}.
$$
Thus, for all $\xi$ in $\Omega$ and all $v\in S_2(\mathfrak{sl}(n,\R))$, we put $\zeta=\xi+^t\psi_v(\phi(\xi))$ such that $C_{\xi'}=C_\xi$.\\

\begin{lem}

\

If $\mathfrak G$ is an overalgebra almost separating for $\mathfrak{sl}(n,\R)$, then for all $\xi$ of $\mathfrak{sl}(n,\R)$ and all $v$ of $S_2(\mathfrak {sl}(n,\R))$, $\zeta=\xi+^t\psi_v(\phi(\xi))$ has the same caracteristic polynomial as $\xi$ and the same eigenvalues.\\
\end{lem}

\begin{proof}
For any matrix $\xi$ of $\mathfrak {sl}(n,\R)$, and all $\varepsilon>0$, there exists $\xi_\varepsilon$ in $\Omega$ such that :
$$
\left\|\xi-\xi_\varepsilon\right\|<\varepsilon.
$$
Since $\xi_\varepsilon$ is in $\Omega$,
$$
\det(\xi_\varepsilon+^t\psi_v(\phi(\xi_\varepsilon))-\lambda I)=\det(\xi_\varepsilon-\lambda I), \quad\forall \lambda,~~\forall v
$$
If $\varepsilon$ tends to $0$, then, for all $\lambda$,
$$
\det(\xi+^t\psi_v(\phi(\xi))-\lambda I)=\det(\xi-\lambda I).
$$
\end{proof}

\begin{thm}

\

For $n>2$, $\mathfrak{sl}(n,\R)$ does not admit an overalgebra almost separating of degree $2$.\\
\end{thm}

\begin{proof}

\

We have seen if $\mathfrak{sl}(n,\R)$ has an overalgebra almost separating of degree $2$, there exists an overalgebra $\mathfrak G$, with
$$
\phi(\xi)=a_0P_0(\xi)+(a_1P_1+\dots+a_4P_4)(\xi\cdot\xi).
$$ 

We will show that, if for all $\xi$ in $\Omega$, and all $v$ in $S_2(\mathfrak{sl}(n,\R))$, $\xi$ and $\zeta=\xi+^t\psi_v(\phi(\xi))$ have the same eigenvalues, then $a_0=a_1=\dots=a_3=0$, and that the function $\phi(\xi)=a_4P_4(\xi\cdot\xi)$ does not separate the coadjoint orbits of $\mathfrak{sl}(n,\R)$, if $n>2$.\\

Taking first $v=U\in\mathfrak{sl}(n,\R)$. Then $\left\langle ^t\psi_U(\phi(\xi)), X\right\rangle=a_0Tr(\xi[X, U])$, and $\zeta=\xi+a_0[U,\xi]$.\\

Let $U=e_{n1}+e_{(n-1)2}$ and $\xi=e_{1n}+e_{2(n-1)}$, thus 
$$
\zeta=a_0(-e_{11}-e_{22}+e_{(n-1)(n-1)}+e_{nn})+e_{1n}+e_{2(n-1)}
$$
and
$$
\det(\zeta-\lambda I)=(-\lambda)^{n-4}(\lambda^2-a_0^2)^2.
$$ 
Therefore, $\zeta$ has the same spectrum as $\xi$ implies $a_0=0$.\\

Put now $v=X.X$, then a direct calculation gives :
$$
^t\psi_{X.X}(\phi(\xi))=\left\{\aligned
&4a_1[X, \xi X \xi]+4a_2Tr(\xi X)[X,\xi]+4a_3[X^2, \xi^2],~~\text{ if }n>3,\\
&4a_1[X, \xi X \xi]+4a_3[X^2, \xi^2],~~\text{ if }n=3.\endaligned\right.
$$

Choose $\xi=e_{1n}$ and $X=e_{n1}$. $\xi$ and $X$ are nilpotent matrices : $X^2=\xi^2=0$ and we obtain :
$$
\xi X=e_{11},\quad [X, \xi X\xi]=e_{nn}-e_{11},\quad Tr(\xi X)=1,\quad [X, \xi]=e_{nn}-e_{11},
$$
thus 
$$
^t\psi_{X.X}(\phi(\xi))=\left\{\aligned&(4a_1+4a_2)(-e_{11}-e_{22}+e_{(n-1)(n-1)}+e_{nn}),~~\text{ if }n>3,\\
&4a_1(e_{33}-e_{22}),~~\text{ if }n=3,\endaligned\right.
$$ 
and, if we note $\zeta=\xi+^t\psi_{X.X}(\phi(\xi))$, then
$$
\det(\zeta-\lambda I)=\left\{\aligned&(-\lambda)^{n-2}(\lambda^2-(4a_1+4a_2)^2),~~\text{ if }n>3,\\
&-\lambda(\lambda^2-(4a_1)^2),~~\text{ if }n=3.\endaligned\right.
$$
Since $\det(\xi-\lambda I)=(-\lambda)^n$, then we deduce that $4a_1+4a_2=0$ if $n>3$ and $a_1=0$ if $n=3$.\\

Suppose now $n>3$. We choose $\xi=e_{1n}+e_{2(n-1)}$ and $X=^t\xi=e_{n1}+e_{(n-1)2}$. These matrices are nilpotent and we obtain 
$$\aligned
\xi^2&=0,\quad &\xi X&=e_{11}+e_{22},\\
\xi X\xi&=e_{1n}+e_{2(n-1)},\quad &[X, \xi X\xi]&=-e_{11}-e_{22}+e_{(n-1)(n-1)}+e_{nn},\\
Tr(\xi X)&=2,\quad &[X, \xi]&=-e_{11}-e_{22}+e_{(n-1)(n-1)}+e_{nn}.
\endaligned
$$
Therefore
$$
^t\psi_{X.X}(\phi(\xi))=(4a_1+8a_2)(-e_{11}-e_{22}+e_{(n-1)(n-1)}+e_{nn})
$$ 
and
$$
\det(\zeta-\lambda I)=(-\lambda)^{(n-4)}(\lambda^2-(4a_1+8a_2)^2)^2.
$$
Since $\det(\xi-\lambda I)=(-\lambda)^n$, we deduce that $4a_1+8a_2=0$. Thus, in all cases, $a_1=a_2=0$.\\

Choose now $\xi=e_{1(n-1)}+e_{(n-1)n}$ and $X=^t\xi=e_{(n-1)1}+e_{n(n-1)}$, then
$$
X^2=e_{n1},\quad \xi^2=e_{1n},\quad [X^2, \xi^2]=-e_{11}+e_{nn}.
$$ 
Therefore
$$
^t\psi_{X^2}(\phi(\xi))=4a_3(-e_{11}+e_{nn})
$$
and
$$
\det(\zeta-\lambda I)=(-\lambda)^{(n-2)}(\lambda^2-(4a_3)^2).
$$
Hence, since $\det(\xi-\lambda I)=(-\lambda)^n$, then $a_3=0$.\\

Thus, we deduce that :
$$
\phi(\xi)=b_2(\xi.\xi)=a_4P_4(\xi.\xi)\quad\text{ or }\quad \left\langle \phi(\xi), U+X.Y\right\rangle=a_4Tr(\xi^2)Tr(XY).
$$
But the overalgebra $\mathfrak G$ is not separating, since, for all $t$ in $[0,1]$, and all matrix $M=D(c_4,\dots,c_n)\in\mathfrak{sl}(n-3,\R)$, with $|c_k|>2$, we define the matrix
$$
\xi_t=\left(\begin{matrix}\frac{1}{2}(t+ \sqrt{4-3t^2})&&&\\ & \frac{1}{2}(t-\sqrt{4-3t^2})&&\\ &&-t&\\ &&&M\end{matrix}\right).
$$
For all $t$, 
$$
\xi_t\in \Omega,~~\det(\xi_t)=t(1-t^2)\prod_k c_k~  \text{ and }~Tr(\xi_t^2)=2+\sum_k c_k^2.
$$
i.e, for all $t$,
$$
\overline{Conv}\left(\Phi(Coad SL(n,\R)\xi_t)\right)=\mathfrak{sl}(n,\R)^\star\times\{U+X\cdot Y\mapsto a_4(2+\sum_k c_k^2)Tr(XY)\},
$$
therefore if $t\neq\frac{1}{\sqrt{3}}$,
$$
Coad ~SL(n,\R)(\xi_t)\neq Coad~ SL(n,\R)(\xi_{\frac{1}{\sqrt{3}}}).
$$
Since $t(1-t^2)< \frac{1}{\sqrt{3}}(1-\frac{1}{3})=\frac{2}{3\sqrt{3}}$, $\det(\xi_t)\neq\det(\xi_{\frac{1}{\sqrt{3}}})$, thus $\xi_t$ is not in the orbit $Coad~SL(n,\R)\xi_{\frac{1}{\sqrt{3}}}$.\\
\end{proof}

\begin{rema}

\

Recall that, if $n=2$, the overalgebra $(\mathfrak{sl}(2,\R)\rtimes\R,[\xi\mapsto Tr(\xi^2)])$ is an overalgebra almost separating of degree $2$ of $\mathfrak{sl}(2,\R)$ (cf. \cite{ASZ} where we use the function $\det(\xi)=- \frac{1}{2}Tr(\xi^2)$).\\

Similarly, $\mathfrak{sl}(3,\R)$ does not admit an overalgebra almost separating of degree $2$ but $\mathfrak{sl}(3,\R)$ admits an overalgebra almost separating of degree $3$.\\

\end{rema}

In this following section, we will show that $\mathfrak{sl}(4,\R)$ does not admit an overalgebra almost separating of degree 2 or 3 but it admits one overalgebra almost separating of degree 4.\\


\section{The case $n=4$ and $p=3$}


\

As above, we shall first find the explicit decomposition of $S^3(\mathfrak{sl}(4,\R))$.\\


\subsection{Decomposition of $S^3(\mathfrak{sl}(4,\R))$}


\
 
We have seen that the module $S^3(\mathfrak{sl}(4,\R))$ is self dual. Then, if the submodule $\Gamma_{a_1a_2a_3}$ appears in the decomposition of $S^3(\mathfrak{sl}(4,\R))$, the submodule $\Gamma_{a_3a_2a_1}\simeq (\Gamma_{a_1a_2a_3})^s$ appears also.\\

The module $S^3(\mathfrak{sl}(4,\R))$ is a submodule of $S^2(\mathfrak{sl}(4,\R))\otimes \mathfrak{sl}(4,\R)$. The decomposition of $S^2(\mathfrak{sl}(4,\R))\otimes \mathfrak{sl}(4,\R)$ is given by Littlewood-Richardson's rule ({\sl{cf.}} \cite{FH}), as follows :
$$\aligned
S^2(\mathfrak{sl}(4,\R))\otimes \mathfrak{sl}(4,\R)
&=(\Gamma_{303}+\Gamma_{212}+\Gamma_{202}+\Gamma_{101})+(\Gamma_{121}+\Gamma_{202}+\Gamma_{101}+\Gamma_{311}+\Gamma_{113})\\
&\hskip 0.5cm +3(\Gamma_{210}+\Gamma_{012})+2\Gamma_{020}+3\Gamma_{101}+\Gamma_{000}.
\endaligned
$$

The highest weight vectors which appear in $S^2(\mathfrak{sl}(4,\R))$ are $v_{202}$, $v_{020}$, $v_{101}$ and $v_{000}$. We deduce that there are 4 highest weight vectors in $S^3(\mathfrak{sl}(4,\R))$ which are $w_{303}=v_{202}.e_{14}$, $w_{121}=v_{020}.e_{14}$, $w_{202}=v_{101}.e_{14}$ and $w_{101}=v_{000}.e_{14}$. These vectors are the highest weight vectors for the simple modules $\Gamma_{303}$, $\Gamma_{121}$, $\Gamma_{202}$ and $\Gamma_{101}$.\\

In $S^2(\mathfrak{sl}(4,\R))\otimes \mathfrak{sl}(4,\R)$, the highest weight vectors  $v_{020}\otimes e_{14}-v_{020}.e_{14}$, $v_{101}\otimes e_{14}-v_{101}.e_{14}$ and $v_{000}\otimes e_{14}-v_{000}.e_{14}$ appear also. The corresponding simple modules of these vectors are, respectively, $\Gamma_{121}$, $\Gamma_{202}$ and $\Gamma_{101}$. Since these vectors are not symmetric, then their corresponding modules are not submodules of $S^3(\mathfrak{sl}(4,\R))$.\\

The highest weight vector of $\Gamma_{311}$ is $e_{14}\otimes e_{13}\otimes e_{14}-e_{14}.e_{13}.e_{14}$ which is not symmetric, then $\Gamma_{311}$ does not appear in $S^3(\mathfrak{sl}(4,\R))$, and $\Gamma_{113}$ does not appear also.

We conclude:
$$
\Gamma_{303}+\Gamma_{121}+\Gamma_{202}+\Gamma_{101}\subset S^3(\mathfrak{sl}(4,\R)).
$$
The additional invariant space of $(\Gamma_{303}+\Gamma_{212}+\Gamma_{202}+\Gamma_{101})$ in $S^3(\mathfrak{sl}(4,\R))$ has the following decomposition, by using the dimensions :
$$
S^3(\mathfrak{sl}(4,\R))/(\Gamma_{303}+\Gamma_{212}+\Gamma_{202}+\Gamma_{101})=(\Gamma_{210}+\Gamma_{012})+\Gamma_{101}+\Gamma_{000}.
$$
Therefore :
$$
S^3(\mathfrak{sl}(4,\R))=(\Gamma_{303}+\Gamma_{212}+\Gamma_{202}+\Gamma_{101})+(\Gamma_{210}+\Gamma_{012}+\Gamma_{101}+\Gamma_{000}).
$$

The highest weight vectors  $w_{303}$, $w_{121}$, $w_{202}$ and $w_{101}$ are:
$$\aligned
w_{303}&=e_{14}.e_{14}.e_{14},\\
w_{121}&=e_{24}.e_{13}.e_{14}-e_{23}.e_{14}.e_{14},\\
w_{202}&=e_{12}.e_{24}.e_{14}+e_{13}.e_{34}.e_{14}+\frac{1}{2}((e_{11}-e_{22})-(e_{33}-e_{44})).e_{14}.e_{14},\\
w_{101}&=8(e_{12}.e_{21}.e_{14}+e_{13}.e_{31}.e_{14}+e_{14}.e_{41}.e_{14}+e_{23}.e_{32}.e_{14}+e_{34}.e_{43}.e_{14})+\\
&\hskip 0.5cm+3(e_{11}.e_{11}.e_{14}+e_{22}.e_{22}.e_{14}+e_{33}.e_{33}.e_{14}+e_{44}.e_{44}.e_{14})-\\
&\hskip 0.5cm-2(e_{11}.e_{22}.e_{14}+e_{11}.e_{33}.e_{14}+e_{11}.e_{44}.e_{14}+e_{22}.e_{33}.e_{14}+e_{22}.e_{44}.e_{14}+e_{33}.e_{44}.e_{14}).
\endaligned
$$

Now looking for the highest weight vectors of the four remaining simple modules.\\

By Littlewood-Richardson's rule, $(\Gamma_{210}+\Gamma_{012})$ appears in the tensorial product $\Gamma_{020}\otimes\Gamma_{101}$ where $\Gamma_{020}$ is in $S^2(\mathfrak{sl}(4,\R))$, and $\Gamma_{101}$ is in $\mathfrak{sl}(4,\R)$.\\

The highest weight vector of the module $\Gamma_{020}$ is :
$$
v_{020}=e_{24}.e_{13}-e_{23}.e_{14}.
$$

We deduce also two other vectors of $\Gamma_{020}$ given by :
$$\aligned
ad_{e_{42}}v_{020}&=((e_{44}-e_{22}).e_{13}-e_{43}.e_{14}+e_{23}e_{12}),\\
ad_{e_{32}} v_{020}&=(e_{34}.e_{13} -e_{24}.e_{12}-(e_{33}-e_{22}).e_{14}).
\endaligned
$$

Thus, there is a highest weight vector of $\Gamma_{210}$, defined by :
$$
w_{210}=e_{12}.e_{24}.e_{13}-e_{12}.e_{23}.e_{14}-e_{14}.e_{43}.e_{14}+e_{13}.e_{34}.e_{13}+(e_{44}-e_{33}).e_{13}.e_{14}.
$$
$w_{210}$ is a non zero vector and its weight is $4L_1+2L_2+2L_3=2\omega_1+2\omega_3$. Indeed :
$$
ad_{e_{12}} w_{210}=0,\quad ad_{e_{23}} w_{210}=0~~~\text{ and }~~ad_{e_{34}} w_{210}=0.
$$

Using the application $s$, the highest weight vector of the module $\Gamma_{012}$ is $v_{210}^s$ or :
$$
w_{012}=e_{34}.e_{13}.e_{24}-e_{34}.e_{23}.e_{14}-e_{14}.e_{21}.e_{14}+e_{24}.e_{12}.e_{24}+(e_{11}-e_{22}).e_{24}.e_{14}.
$$

It remains the modules $\Gamma_{101}$ and $\Gamma_{000}$ which appear in the tensorial product $\Gamma_{101}\otimes\Gamma_{101}$. The first factor is in $S^2(\mathfrak{sl}(4,\R))$, the second is in $\mathfrak{sl}(4,\R)$.\\

There is a basis for the first factor defined by the following vectors :
$$
e'_{ij}=e_{i1}.e_{1j}+e_{i2}.e_{2j}+e_{i3}.e_{3j}+e_{i4}.e_{4j}-\frac{1}{2}(e_{11}+e_{22}+e_{33}+e_{44}).e_{ij}.
$$

In $S^2(\mathfrak{sl}(4,\R))\subset\Gamma_{101}\otimes\Gamma_{101}$, we have seen that the corresponding highest weight vectors are:
$$\aligned
v_{101}&=e_{12}.e_{24}+e_{13}.e_{34}+\frac{1}{2}((e_{11}-e_{22})-(e_{33}-e_{44})).e_{14},\\
v_{000}&=8\sum_{1\leq i<j\leq 4}e_{ij}.e_{ji}+\sum_{1\leq i<j\leq 4}(e_{ii}-e_{jj}).(e_{ii}-e_{jj}).
\endaligned
$$
By replacing the first factor $e_{ij}$ by the factor $e'_{ij}$, we obtain the highest weight vectors $w'_{101}$ and $w_{000}$ ( $w_{000}$ is not developed) :
$$\aligned
w'_{101}&=2e_{12}(2e_{23}.e_{34}+2e_{21}.e_{14}+(e_{22}-e_{11}).e_{24}-(e_{33}-e_{44}).e_{24})+\\
&\hskip 0.5cm +2e_{13}(2e_{32}.e_{24}+2e_{31}.e_{14}+(e_{33}-e_{11}).e_{34}+(e_{44}-e_{22}).e_{34})+\\
&\hskip 0.5cm +(e_{11}-e_{22})(2e_{12}.e_{24}+2e_{13}.e_{34}+(e_{11}-e_{22}).e_{14}-\\
&\hskip 0.5cm -(e_{33}-e_{44}).e_{14})-(e_{33}-e_{44})(2e_{12}.e_{24}+2e_{13}.e_{34}+(e_{11}-e_{22}).e_{14}-(e_{33}-e_{44}).e_{14}),\\
w_{000}&=4(e_{12}.e'_{21}+e'_{12}.e_{21}+e_{13}.e'_{31}+e'_{13}.e_{31}+e_{14}.e'_{41}+e'_{14}.e_{41}+\\
&\hskip 0.5cm +e_{23}.e'_{32}+e'_{23}.e_{32}+e_{24}.e'_{42}+e'_{24}.e_{42}+e_{34}.e'_{43}+e'_{34}.e_{43})+\\
&\hskip 0.5cm +(e_{11}-e_{22})(e'_{11}-e'_{22})+(e_{11}-e_{33})(e'_{11}-e'_{33})+(e_{11}-e_{44})(e'_{11}-e'_{44})+\\
&\hskip 0.5cm +(e_{22}-e_{33})(e'_{22}-e'_{33})+(e_{22}-e_{44})(e'_{22}-e'_{44})+(e_{33}-e_{44})(e'_{33}-e'_{44}).
\endaligned
$$

\


\subsection{Trace forms and intertwining of $S^3(\mathfrak{sl}(4,\R))$}


\

As for $S^2(\mathfrak{sl}(n,\R)$, we know 12 trace forms. Denote by $\xi$, $\eta$ and $\zeta$ elements in $\left(\mathfrak{sl}(4,\R)\right)^\star=\mathfrak{sl}(4,\R)$, and $X$, $Y$, $Z$ elements in $\mathfrak{sl}(4,\R)$. The trace forms are the following :
$$\begin{array}{rlrlrl}
T_1&=Tr(\xi\eta\zeta XYZ),&T_2&=Tr(\xi\eta X\zeta YZ),&T_3&=Tr(\xi\eta XY\zeta Z),\\
T_4&=Tr(\xi X\eta Y\zeta Z),&T_5&=Tr(\xi\eta\zeta X)Tr(YZ),&T_6&= Tr(\xi\eta XY)Tr(\zeta Z),\\
T_7&=Tr(\xi XYZ)Tr(\eta \zeta ),& T_8&=Tr(\xi X\eta Y)Tr(\zeta Z),& T_9&=Tr(\xi\eta\zeta)Tr( XYZ),\\
T_{10}&=Tr(\xi\eta X)Tr( \zeta YZ),& T_{11}&=Tr(\xi\eta)Tr(\zeta X)Tr(YZ),& T_{12}&=Tr(\xi X)Tr(\eta Y)Tr(\zeta Z).
\end{array}
$$

Recall that, in the previous section, we calculated the 8 highest weight vectors of the decomposition of $S^3(\mathfrak{sl}(4,\R))$, i.e the free system
$$
(w^1,\dots,w^8)=(w_{303}, w_{121}, w_{202}, w_{210}, w_{012}, w_{101}, w'_{101}, w_{000}).
$$

Let $M$ the matrix with 8 rows and 12 columns whose entries are the numbers $\langle T_i(w^k),(w^k)^t\rangle$ ($i=1,\dots, 12$, $k=1,\dots,8$) where the vector $e_{j_1i_1}.e_{j_2i_2}.e_{j_3i_3}$ of $S^3(\mathfrak{sl}(4,\R))$ is noted $(e_{i_1j_1}.e_{i_2j_2}.e_{i_3j_3})^t$ .\\

We obtain, by using a symbolic computation program, the following matrix:
$$
M=\left(\begin{array}{cccccccccccc}
0& 0& 0& 36& 0& 0& 0& 36& 0&  0&  0&  36  \\
0& 0& 0& 0& 0&  0&  0&  4&  0&   0&  0&  12\\
0& 1& 1&  3&  0&  1&  0&  4&  0&   1&  0&  6\\
0&  0&   4&   0&  0&  4&  0&   4&  0&  0&  0&  6\\
0&  1&  0& 0&  0&  1&  0&  2&  0&   0&  0&  6\\
1&   1&  2&   0&  2&  3&  2&  2&  0&   0&  4&  6\\
1&   0&  0&   0&  0&  2&  0&  0&  0&   0&  0&  6\\
3&   0&  0&   0&  0&  3&  0&  0&  9&   0&  0&  6
\end{array}\right)
$$
The rank of this matrix is $8$.\\

We extract the columns 1, 2, 3, 4, 5, 8, 10, 9, so we obtain the following intertwining. Explicitly :
$$
\aligned
\left\langle P_1(\xi\eta\zeta),   XYZ\right\rangle&=\text{Sym}(Tr(\xi\eta\zeta XYZ)),\\
\left\langle P_2(\xi\eta\zeta),   XYZ\right\rangle&=\text{Sym}(Tr(\xi\eta X\zeta YZ)),\\
\left\langle P_3(\xi\eta\zeta),   XYZ\right\rangle&=\text{Sym}(Tr(\xi\eta X Y\zeta Z)),\\
\left\langle P_4(\xi\eta\zeta),  XYZ\right\rangle&=\text{Sym}(Tr(\xi X\eta Y\zeta Z)),\\
\left\langle P_5(\xi\eta\zeta),  XYZ\right\rangle&=\text{Sym}(Tr(\xi \eta \zeta X)Tr(YZ)),\\
\left\langle P_6(\xi\eta\zeta),  XYZ\right\rangle&=\text{Sym}(Tr(\xi X\eta Y)Tr(\zeta Z)),\\
\left\langle P_7(\xi\eta\zeta),  XYZ\right\rangle&=\text{Sym}(Tr(\xi\eta X)Tr( \zeta YZ)),\\
\left\langle P_8(\xi\eta\zeta),  XYZ\right\rangle&=\text{Sym}(Tr(\xi\eta\zeta)Tr( XYZ)).\\
\endaligned
$$ 
The notation '\text{Sym}' means that the expression is symmetrical in $\xi$, $\eta$, $\zeta$.\\

If $N$ is the sub-matrix of $M$, with 8 rows and 8 columns whose entries are $\langle P_i(w^k),(w^k)^t\rangle$, $i=1,\dots,8$, then
$$
N=\left(\begin{array}{cccccccccccc}
0& 0& 0& 36&  0& 36&  0& 0  \\
0& 0& 0& 0 &  0&  4&  0& 0  \\
0& 1& 1& 3 &  0&  4&  1& 0  \\
0& 0& 4& 0 &  0&  4&  0& 0  \\
0& 1& 0& 0 &  0&  2&  0& 0  \\
1& 1& 2& 0 &  2&  2&  0& 0  \\
1& 0& 0& 0 &  0&  0&  0& 0  \\
3& 0& 0& 0 &  0&  0&  0& 9
\end{array}\right)
$$
The rank of this matrix is also 8. Thus $(P_1, P_2, P_3, P_4, P_5, P_6, P_7, P_8)$ are independent. Therefore:

\begin{lem}

\

The applications $P_i:S^3(\mathfrak{sl}(4,\R)^\star)~\longrightarrow~\left(S^3(\mathfrak{sl}(4,\R))\right)^\star$ defined above, form a basis of the space of intertwining of the module $S^3(\mathfrak{sl}(4,\R))$.\\

\end{lem}


\subsection{$\mathfrak {sl}(4,\R)$ does not admit an overalgebra almost separating of degree 3}


\

\begin{thm}

\

The algebra $\mathfrak{sl}(4,\R)$ does not admit an overalgebra almost separating of degree 3.
\end{thm}

\begin{proof}

\

We have seen that if $\mathfrak{sl}(4,\R)$ admits an overalgebra of degree 3, then $\mathfrak{sl}(4,\R)$ admits an overalgebra of the form
$$
(\mathfrak{sl}(4,\R)\rtimes S_3(\mathfrak{sl}(4,\R)),(b_1+b_2+b_3)\circ\tau).
$$
In this case, $b_i$ are an intertwining, $b_1=0$, $\langle b_2(\xi.\eta),X.Y\rangle=aTr(\xi\eta)Tr(XY)$ and $b_3$ is written :
$$
b_3(\xi.\eta.\zeta)=\sum_{j=1}^8c_j P_j(\xi.\eta.\zeta).
$$

Then, we choose $v=X.X.X$ in $S_3(\mathfrak{sl}(4,\R))$ and we calculate $^t\psi_v( P_j(\xi.\xi.\xi))$. Explicitly:
$$\aligned
^t\psi_{v}(P_1(\xi.\xi.\xi))&=[X^3, \xi^3],\\
^t\psi_{v}(P_2(\xi.\xi.\xi))&=[X^2\xi^2X, \xi]+[X\xi X^2, \xi^2],\\
^t\psi_{v}(P_3(\xi.\xi.\xi))&=[X^2\xi X, \xi^2]+[X\xi^2 X^2, \xi],\\
^t\psi_{v}(P_4(\xi.\xi.\xi))&=3[X\xi X\xi X, \xi],\\
^t\psi_{v}(P_5(\xi.\xi.\xi))&=Tr(X^2)[X, \xi^3],\\
^t\psi_{v}(P_6(\xi.\xi.\xi))&=2Tr(\xi X)[X\xi X, \xi]+Tr(\xi X\xi X)[X, \xi],\\
^t\psi_{v}(P_7(\xi.\xi.\xi))&=Tr(\xi X^2)[X, \xi^2] +Tr(\xi^2 X)[X^2, \xi],\\
^t\psi_{v}(P_8(\xi.\xi.\xi))&=0.
\endaligned
$$

Let $\xi=e_{14}$, then $\xi^2=0$ and $^t\psi_{v}(P_j(\xi.\xi.\xi))=0$, for $j=1$, 2, 3, 5, 7.\\

Let now $X=e_{14}+e_{41}$, then $\xi X=e_{11}$ and $X\xi X \xi X=X-\xi=e_{41}$.\\

So, we obtain $^t\psi_{v}(P_4(\xi.\xi.\xi))=-3(e_{11}-e_{44})$ and $^t\psi_{v}(P_6(\xi.\xi.\xi))=3(e_{11}-e_{44})$. Thus, with same notations as above, 
$$
\zeta=\xi+^t\psi_{v}(\phi(\xi))=3(c_6-c_4)(e_{11}-e_{44})+e_{14}.
$$ 
Therefore, 
$$
\det(\zeta-\lambda I)=\lambda^2(\lambda^2-9(c_6-c_4)^2).
$$
We deduce the relation $c_6-c_4=0$.\\

On the other hand, let $X=e_{14}-e_{41}$, then :
$$
\begin{array}{rlrlrl}
\xi X&=-e_{11},&  \xi X\xi X&=e_{11},&  X\xi X\xi X&=-e_{41},\cr
X\xi X&=e_{41},& [X\xi X\xi X, \xi]&=e_{11}-e_{44},&  X\xi X&=e_{41},\cr
[X\xi X,\xi]&=-e_{11}+e_{44},&  [X, \xi]&=e_{11}-e_{44},&  \xi X\xi X&=e_{11}.
\end{array}
$$

Thus, we get $^t\psi_{v}(P_6(\xi.\xi.\xi))=3(e_{11}-e_{44})$ and $^t\psi_{v}(P_4(\xi.\xi.\xi))=3(e_{11}-e_{44})$. Therefore, if $\zeta=\xi+^t\psi_{v}(\phi(\xi))$,
$$
\det(\zeta-\lambda I)=\lambda^2(\lambda^2-9(c_6+c_4)^2)
$$
then $c_6+c_4=0$. This shows that $c_6=c_4=0$.\\

We choose now $\xi=e_{13}+e_{34}$, and $X=^t\xi=e_{31}+e_{43}$. Then:
$$\begin{array}{rlrlrlrl}
X^2&=e_{41}, &X^2\xi^2X&=e_{43}, &[\xi, X^2\xi^2X]&=e_{33}-e_{44}, &X\xi X^2&=e_{41},\cr
[\xi^2, X\xi X^2 ]&=e_{11}-e_{44}, &X^2\xi X&=e_{41}, &[\xi^2, X^2\xi X]&=e_{11}-e_{44}, &X\xi^2 X^2&=e_{31},\cr
[\xi, X\xi^2X^2]&=e_{11}-e_{33}, &X^2\xi^2X&=e_{43}, &[X^2\xi^2X, \xi]&=e_{44}-e_{33}, &X\xi X^2&=e_{41},\cr
[X\xi X^2, \xi^2]&=e_{44}-e_{11}.&&&&&
\end{array}
$$
We deduce that $^t\psi_{v}(P_2(\xi.\xi.\xi))=e_{11}+e_{33}-2e_{44}$ and $^t\psi_{v}(P_3(\xi.\xi.\xi))=2e_{11}-e_{33}-e_{44}$. Therefore, 
$$
\zeta=\xi+^t\psi_{v}(\phi(\xi))=e_{13}+e_{34}+(c_2+2c_3)e_{11}+(c_2-c_3)e_{33}-(2c_2+c_3)e_{44}
$$
and
$$
\det(\zeta-\lambda I)=-\lambda(c_2+2c_3-\lambda)(c_2-c_3-\lambda)(-2c_2-c_3-\lambda)
$$
Hence, the spectrum of $\zeta$ is the same as $\xi$, i.e $\{0\}$ implies $c_2+2c_3=0$, $c_2-c_3=0$, and $2c_2+c_3=0$, so $c_2=c_3=0$.\\ 

Now, let $\xi=e_{13}+e_{14}+e_{34}$ and $X=^t\xi=e_{31}+e_{41}+e_{43}$, then
$$\begin{array}{rlrlrl}
\xi^2&=e_{14}, &X^2&=e_{41}, &\xi X^2&=e_{11}+e_{31},\cr
\xi^2X&=e_{11}+e_{13}, &[\xi^2, X]&=e_{11}+e_{13}-e_{34}-e_{44}, &[\xi, X^2]&=e_{11}+e_{31}-e_{43}-e_{44}.
\end{array}
$$
Thus,
$$
^t\psi_{v}(P_7(\xi.\xi.\xi))=2e_{11}+e_{13}+e_{31}-e_{34}-e_{43}-2e_{44}
$$
and, if $\zeta=\xi+^t\psi_{v}(\phi(\xi))$,
$$
\det(\zeta-\lambda I)=-\lambda(-\lambda^3+\lambda(5c_7^2+c_7)+2c_7^3+c_7^2+2c_7).
$$
Therefore, the spectrum of $\zeta$ is the same as $\xi$, i.e $\{0\}$ implies $c_7=0$.\\

Later, we choose $\xi=e_{12}+e_{23}+e_{34}$ and $X=^t\xi$. Then, $\xi^2=e_{13}+e_{24}$, $\xi^3=e_{14}$, $X^3=e_{41}$. Thus $^t\psi_{v}(P_1(\xi.\xi.\xi))=e_{11}-e_{44}$ and the spectrum of $\zeta=\xi+^t\psi_{v}(\phi(\xi))$ is $\{0\}$ implies $c_1=0$.\\

Finally, we choose another $X=e_{14}+e_{41}$ and we allowed $\xi=e_{12}+e_{23}+e_{34}$. Then $X^2=e_{11}+e_{44}$ and $\xi^3=e_{11}-e_{44}$. Therefore, $^t\psi_{v}(P_5(\xi.\xi.\xi))=2(e_{11}-e_{44})$ and $\det(\zeta-\lambda I)=\lambda^4$ implies $c_5=0$.\\

We finally get:
$$
\langle\phi(\xi),U+X.Y+X'.Y'.Z'\rangle=a_4Tr(\xi^2)Tr(XY)+c_8Tr(\xi^3)Tr(X'Y'Z').
$$

But we consider, for $0<t<1$, the matrices
$$
\xi_t=\left(\begin{matrix} \sqrt{1+t}&&&\\ &-\sqrt{1+t}&&\\&& \sqrt{1-t}&\\ &&& - \sqrt{1-t} \end{matrix}\right).
$$

$\xi_t$ is an element of $\Omega$ for all $t$, $Tr(\xi_t^2)=4$ and $Tr(\xi_t^3)=0$ for all $t$. Although, $\det(\xi_t)=(1-t^2)^2$. Therefore, with the same argument as in a previous section, we have, for all $t$,
$$
\overline{Conv}\left(\Phi(Coad~SL(4,\R)\xi_t)\right)=\left(\mathfrak{sl}(4,\R)\right)^\star\times\{U+X.Y+X'.Y'.Z'\mapsto 4a_4Tr(XY)\}.
$$
But, if $t\neq\frac{1}{2}$, $\xi_t$ is not in the orbit $Coad~SL(4,\R)\xi_{\frac{1}{2}}$.

Thus, $\mathfrak{sl}(4,\R)$ does not admit an overalgebra almost separating of degree 3.\\

\end{proof}

In fact, we think that the following conjecture is always true :

\

\begin{conj}

For all $n$, $\mathfrak{sl}(n,\R)$ does not admit an overalgebra almost separating of degree $n-1$, but it admits an overalgebra almost separating of degree $n$.\\

More generally, if $\mathfrak g$ is a real and deployed semi simple Lie algebra and if $k$ is the greatest degree of the generators of the algebra of invariant functions on $\mathfrak g$, then $\mathfrak g$ admits an overalgebra almost separating of degree $k$. But $\mathfrak g$ does not admit an overalgebra almost separating of degree $k-1$.

\end{conj}

The hypothesis `$\mathfrak g$ deployed' is necessary. Indeed, we remark that $\mathfrak{sl}(2,\R)$ does not admit an overalgebra almost separating of degree 1, but the Lie algebra $\mathfrak{su}(2)$ admits an overalgebra almost separating of degree 1 since its adjoint orbits are spheres which are characterized by the closure of their convex hull.\\


\end{document}